\documentclass[12pt]{amsart}

\usepackage[top=1.0in, bottom=1.0in, left=1in, right=1in]{geometry}

\usepackage[backref]{hyperref}

\usepackage{amssymb, amsmath, amsthm, mathrsfs, comment} 

\usepackage{latexsym}
\usepackage{hyperref}
\usepackage{enumerate}

\newtheorem{theorem}{Theorem}[section]
\newtheorem{lemma}[theorem]{Lemma}
\newtheorem{proposition}[theorem]{Proposition}

\newtheorem{question}[theorem]{Question}

\newtheorem{corollary}[theorem]{Corollary}

\theoremstyle{remark}

\numberwithin{equation}{section}

\newcommand{\bbZ}{\ensuremath{\mathbb{Z}}}
\newcommand{\bbQ}{\ensuremath{\mathbb{Q}}}
\newcommand{\psl}{\ensuremath{\operatorname{PSL}}}
\newcommand{\bbF}{\ensuremath{\mathbb{F}}}
\newcommand{\mo}{{-1}}
\newcommand{\sln}{\ensuremath{\operatorname{SL}}}

\begin{document}

\title{On non-surjective word maps on $\mathrm{PSL}_{2}(\mathbb{F}_{q})$}

\author{Arindam Biswas}
\address{Department of Mathematics, Technion - Israel Institute of Technology, Haifa 32000, Israel}
\curraddr{}
\email{biswas@campus.technion.ac.il}
\thanks{}

\author{Jyoti Prakash Saha}
\address{Department of Mathematics, Indian Institute of Science Education and Research Bhopal, Bhopal Bypass Road, Bhauri, Bhopal 462066, Madhya Pradesh,
India}
\curraddr{}
\email{jpsaha@iiserb.ac.in}
\thanks{}



\begin{abstract}
Jambor--Liebeck--O'Brien showed that there exist non-proper-power word maps which are not surjective on $\mathrm{PSL}_{2}(\mathbb{F}_{q})$ for infinitely many $q$. This provided the first counterexamples to a conjecture of Shalev which stated that if a two-variable word is not a proper power of a non-trivial word, then the corresponding word map is surjective on $\mathrm{PSL}_2(\mathbb{F}_{q})$ for all sufficiently large $q$. Motivated by their work, we construct new examples of these types of non-surjective word maps. As an application, we obtain non-surjective word maps on the absolute Galois group of $\mathbb Q$. 
\end{abstract}

\maketitle 

\section{Introduction}

A word in $k$ variables is an expression of the form
$$w(x_{1},\ldots, x_{k}) = \prod_{j=1}^{t}x_{i_{j}}^{\varepsilon_{j}},$$ where $i_{j}\in [1,k]$, for each $j \in [1,t]$ and $\varepsilon_{j} = \pm 1$. Given a word $w$ in $k$ variables, and a group $G$, one has the verbal mapping 
$$w: G\times \cdots \times G \rightarrow G,$$
defined by, 
$$w(g_{1}, \ldots, g_{k}) = \prod_{j=1}^{t}g_{i_{j}}^{\varepsilon_{j}}.$$
See Segal \cite[Chapter 1]{SegalBook}. It was first shown by Liebeck--Shalev \cite[Theorem 1.6]{LiebeckShalev}, that for a given non-trivial word $w$, each element of each large enough finite simple group $G$ can be expressed as a product of $c(w)$ values of $w$ in $G$ unless $w$ is trivial on $G$, with $c(w)$ only depending on $w$. In recent years, a lot of research has been devoted to studying the $c(w)$'s. For instance, it has since been established that $c(w) = 2$. This follows from the works of Larsen--Shalev \cite{LarsenShalev}, Shalev \cite{Shalev}, Larsen--Shalev--Tiep \cite{LarsenShalevTiep}. Some words actually have $c(w) = 1$ (i.e., they are surjective) and in fact, it was a long-standing conjecture to show that the commutator word has $c(w) = 1$ (the Ore's conjecture). This was resolved by Liebeck--O'Brien--Shalev--Tiep \cite{LiebeckBrienShalevTiep}. There are other words which are surjective. On the other hand, it can be easily seen that, if we take $w = x_{1}^{n}$ and $G$ is a finite simple group with $\mathrm{gcd}(|G|,n) > 1$, then $c(w) > 1$. Shalev conjectured that if $w(x_{1}, x_{2})$ is not a proper power of a non-trivial word, then the corresponding
word map is surjective on $\psl_{2}(\bbF_q)$ for all sufficiently large $q$, see \cite[Conjecture 8.3]{BandmanGarionGrunewald}. This conjecture was recently disproved by Jambor--Liebeck--O'Brien \cite[Theorem 1]{JamborLiebeckOBrien}. They gave examples of an infinite family of non-proper-power words which are non-surjective on an infinite family of finite simple groups. The words that they constructed have word lengths $3q-1$ where $q\geq 5$ is a prime. In this note, our main motivations are to extend this class of non-proper-power words and also to construct examples in other word lengths.

\subsection{Statement of results}
In \cite[Theorem 1]{JamborLiebeckOBrien}, for primes $q\geq 5$, the words of the form 
$$x_1^2 [x_1^{-2}, x_2^{-1}]^{\frac{q-1}2}$$
(having word length $3q-1$) have been considered and the non-surjectivity of the induced word map on $\psl_2(\bbF_{p^n})$ has been established for primes $p$ and integers $n$ satisfying certain suitable conditions. Our principal result (Theorem \ref{Thm:NonSurjWord}) considers words that are not in the purview of \cite[Theorem 1]{JamborLiebeckOBrien}. It focuses on the words of the form 
$$x_1^{ 2} (x_1^2 x_2 x_1^{\pm 2}x_2^{-1})^{\frac{r-1}2}$$
(having word length $3r-1$) where $r\geq 5$ is any odd integer not divisible by $3$, and the words of the form
$$x_1^{-2} (x_1^2 x_2 x_1^{\pm 2}x_2^{-1})^{\frac{r-1}2}$$
(having word length $3r-5$) where $r\geq 7$ is any odd integer such that $r+1$ is not divisible by $3$. In Theorem \ref{Thm:NonSurjWord}, we establish the non-surjectivity of the induced word map on $\psl_2(\bbF_{p^n})$ for primes $p$ and integers $n$ satisfying appropriate conditions. 
Next, we show that there exist primes $p$ and integers $n$ such that these conditions hold (see Proposition \ref{Prop:InfManyPrime}). 
In \S \ref{SubSec:NonSurjGalois}, we combine Theorem \ref{Thm:NonSurjWord} and Proposition \ref{Prop:InfManyPrime} with 
the results on the inverse Galois problem for $\psl_2(\bbF_p)$ (obtained by Shih \cite{ShihInverseGaloisPSL2Fp} and Zywina \cite{ZywinaInverseGaloisPSL}) to show that the above-mentioned words are non-surjective on the absolute Galois group of the number fields having degree coprime to $6$, see Proposition \ref{Prop:NonSurjAbsGal}. 

\section{Identities involving trace polynomial of word maps}

For any word $w$ in the free group $F_2$, it was observed by Vogt \cite{VogtENS}, and Fricke and Klein \cite{FrickeKleinAutoFunct} that 
the trace of $w(x, y)$ can be expressed as a polynomial in terms of the traces of $x, y, xy$ for any $2\times 2$ matrices $x, y$ with determinant $1$. More precisely, there exists a polynomial $\tau(w) \in \bbZ[s, t, u]$ such that for any field $K$ and any $x, y\in \sln_2(K)$, the trace of $w(x, y)$ is equal to the polynomial $\tau(w)$ evaluated at $s = \mathrm{tr} (x), t = \mathrm{tr} (y), u = \mathrm{tr} (xy)$. 
For a proof of this result, we refer to the works of Horowitz \cite[Theorem 3.1]{HorowitzFreeRepreSL2}, Plesken--Fabia\'{n}ska \cite[Theorem 2.2]{PleskenFabianskaL2Quot} and Jambor \cite[Theorem 2.2]{JamborL2Quotient}.

Let $y_1$ denote one of the words 
$$x_1^2x_2 x_1^{ 2} x_2^\mo, 
\quad 
x_1^2x_2 x_1^{- 2} x_2^\mo.$$
For $k\in \bbZ$ with $k\neq 1$, 
let $y_k$ denote the $k$-th power of the word $y_1$.

\begin{lemma}
\label{Lemma:ForwardBackward}
If 
\begin{equation}
\label{Eqn:Swap}
\tau(x_1^2 y_{k-1}) 
= \tau (x_1^{-2} y_k)
\end{equation}
holds for $k = a, a+1$ with $a\in \bbZ$, then it holds for $k = a+2$. 
If Equation \eqref{Eqn:Swap} holds for $k = b,  b-1$ with $b\in \bbZ$, then it holds for $k = b-2$. 
For $k = 0, 1$, Equation \eqref{Eqn:Swap} holds. 
Consequently, Equation \eqref{Eqn:Swap} holds for any $k\in \bbZ$. 
\end{lemma}

\begin{proof}

If Equation \eqref{Eqn:Swap} holds for $k = a,  a+1$, then it holds for $k = a+2$ since 
\begin{align*}
\tau(x_1^{-2} y_{a+2}) 
& = \tau (y_{a+2} x_1^{-2}) \\
& = \tau (y_1^2 y_{a} x_1^{-2}) \\
& = \tau(y_1) \tau(y_{a+1}x_1^{-2}) - \tau(y_{a} x_1^{-2}) \\
& = \tau(y_1) \tau(x_1^{-2}y_{a+1}) - \tau(x_1^{-2}y_{a} ) \\
& = \tau(y_1) \tau(x_1^2y_{a}) - \tau(x_1^2 y_{a-1} ) \\
& = \tau(y_1) \tau(y_{a}x_1^2) - \tau(y_{a-1} x_1^2 ) \\
& = \tau(y_1^2 y_{a-1}x_1^2) \\
& = \tau(y_{a+1}x_1^2) \\
& = \tau(x_1^2 y_{a+1}) \\
\end{align*}
hold. 
If Equation \eqref{Eqn:Swap} holds for $k = b,  b-1$, then it holds for $k = b-2$ since 
\begin{align*}
\tau(x_1^{-2} y_{b-2}) 
& = \tau(y_{b-2}x_1^{-2} ) \\
& = \tau(y_1)\tau(y_1 y_{b-2}x_1^{-2} ) - \tau(y_1^2 y_{b-2}x_1^{-2} )\\
& = \tau(y_1)\tau(y_{b-1}x_1^{-2} ) - \tau(y_{b}x_1^{-2} )\\
& = \tau(y_1)\tau(x_1^{-2} y_{b-1}) - \tau(x_1^{-2} y_{b})\\
& = \tau(y_1)\tau(x_1^{2} y_{b-2}) - \tau(x_1^{2} y_{b-1})\\
& = \tau(y_1)\tau(y_{b-2}x_1^{2} ) - \tau(y_{b-1}x_1^{2} )\\
& = \tau(y_1)\tau(y_{b-2}x_1^{2} ) - (\tau(y_{-1}) \tau(y_{-1} y_{b-1}x_1^{2} ) - \tau(y_{-2} y_{b-1} x_1^2))\\
& = \tau(y_1)\tau(y_{b-2}x_1^{2} ) - (\tau(y_1) \tau(y_{b-2}x_1^{2} ) - \tau(y_{b-3} x_1^2))\\
& = \tau(y_{b-3} x_1^2) \\
& = \tau( x_1^2 y_{b-3}) 
\end{align*}
hold. 

Note that Equation \eqref{Eqn:Swap} holds for $k = 1$ since 
\begin{align*}
\tau(x_1^2) 
& = \tau(x_2 x_1^{\pm 2}x_2^\mo) \\
& = \tau(x_1^{-2} y_1).
\end{align*}
It follows that 
\begin{align*}
\tau(x_1^2y_\mo) 
& = \tau(y_\mo^\mo x_1^{-2}) \\
& = \tau(y_1 x_1^{-2}) \\
& = \tau(x_1^{-2} y_1) \\
& = \tau(x_1^2) \\
& = \tau(x_1^{-2}). 
\end{align*}
So, Equation \eqref{Eqn:Swap} holds for $k = 0$. Consequently, Equation \eqref{Eqn:Swap} holds for any $k\in \bbZ$. 
\end{proof}

\begin{lemma}
\label{Lemma:Factorization}
For any integer $k\geq 1$, the equalities 
\begin{align*}
\tau(x_1^{2}y_{\pm k}) 
& = \tau(x_1^{\pm 2} y_k) \\
& = \tau(x_1^2)
\left( 
\sum_{i = 1} ^{k_\pm}
(-1)^{k_\pm -i} \tau(y_i) + (-1)^{k_\pm}
\right)\\
& = 
(\tau(x_1)^2-2)
\prod_{i=1}^{k_\pm} 
(\tau(y_1) + \zeta_{2k_\pm+1}^i + \zeta_{2k_\pm+1}^{-i}) 
\end{align*}
hold where 
$k_\pm 
= k - \frac 12 \pm \frac 12$. 
\end{lemma}

\begin{proof}
The first equality follows since 
\begin{align*}
\tau(x_1^2 y_{-k}) 
& = \tau (y_k x_1^{-2}) \\
& = \tau(x_1^{-2} y_k)
\end{align*}
hold for any integer $k$. 
Note that for any integer $k$,
\begin{align*}
\tau(x_1^2y_k) 
& = \tau(x_1) \tau(x_1y_k) - \tau(y_k) \\
& = \tau(x_1) (\tau(x_1^\mo) \tau(x_1^\mo x_1y_k) - \tau(x_1^{-2} x_1y_k)) - \tau(y_k) \\
& = \tau(x_1) \tau(x_1^\mo) \tau(y_k) - \tau(x_1) \tau(x_1^\mo y_k) - \tau(y_k) \\
& = \tau(x_1)^2\tau(y_k) - \tau(x_1^\mo) \tau(x_1^\mo y_k) - \tau(y_k)\\
& = (\tau(x_1^2) + 2)\tau(y_k) - \tau(x_1^\mo) \tau(x_1^\mo y_k) - \tau(y_k)\\
& = \tau(x_1^2) \tau(y_k) - (\tau(x_1^\mo) \tau(x_1^\mo y_k) - \tau(y_k))\\
& = \tau(x_1^2)\tau(y_k) - \tau(x_1^{-2} y_k) \\
& = \tau(x_1^2)\tau(y_k) - \tau(x_1^2 y_{k-1}) 
\end{align*}
hold where the second last equality follows from Lemma \ref{Lemma:ForwardBackward}. It follows that for any integer $k\geq 1$, 
\begin{align*}
\tau(x_1^2y_k) 
& = \tau(x_1^2y_k) - (-1)^k \tau(x_1^2) + (-1)^k \tau(x_1^2)\\
& = (-1)^k 
((-1)^k \tau(x_1^2y_k) - \tau(x_1^2)) + (-1)^k \tau(x_1^2)\\
& = (-1)^k 
\sum_{i = 1} ^k 
(
(-1)^i \tau(x_1^2y_i) 
- 
(-1)^{i-1} \tau(x_1^2 y_{i-1}) ) + (-1)^k \tau(x_1^2)\\
& = (-1)^k 
\sum_{i = 1} ^k 
(
(-1)^i (\tau(x_1^2y_i) + \tau(x_1^2 y_{i-1}))) + (-1)^k \tau(x_1^2)\\
& = (-1)^k 
\sum_{i = 1} ^k 
((-1)^i \tau(x_1^2) \tau(y_i)) + (-1)^k \tau(x_1^2)\\
& = \tau(x_1^2)
\left( 
\sum_{i = 1} ^k 
(-1)^{k-i} \tau(y_i) + (-1)^k 
\right)
\end{align*}
and 
\begin{align*}
\tau(x_1^2 y_{-k})
& = \tau(x_1^2 y_{-k})  - (-1)^k \tau(x_1^2) + (-1)^k \tau(x_1^2) \\
& = (-1)^k( (-1)^k\tau(x_1^2 y_{-k})  - \tau(x_1^2) ) + (-1)^k \tau(x_1^2) \\
& = (-1)^k \sum_{i=1}^k ( (-1)^i\tau(x_1^2 y_{-i})  - (-1)^{i-1} \tau(x_1^2y_{-(i-1)} ) ) + (-1)^k \tau(x_1^2) \\
& = (-1)^k \sum_{i=1}^k ( (-1)^i(\tau(x_1^2 y_{-i})  + \tau(x_1^2y_{-(i-1)} ) ) )+ (-1)^k \tau(x_1^2) \\
& = (-1)^k \sum_{i=1}^k ( (-1)^i(\tau(x_1^2) \tau(y_{-i+1}) ) )+ (-1)^k \tau(x_1^2) \\
& = \tau(x_1^2)
\left( 
\sum_{i = 1} ^k 
(-1)^{k-i} \tau(y_{-i+1}) + (-1)^k 
\right)\\
& = \tau(x_1^2)
\left( 
\sum_{i = 1} ^k 
(-1)^{k-i} \tau(y_{-i+1}) + (-1)^k 
\right)\\
& = \tau(x_1^2)
\left( 
\sum_{i = 2} ^k 
(-1)^{k-i} \tau(y_{-i+1}) + 2(-1)^{k-1} + (-1)^k 
\right)\\
& = \tau(x_1^2)
\left( 
\sum_{i = 1} ^{k-1} 
(-1)^{(k-1)-i} \tau(y_{-i}) + (-1)^{k-1}
\right)\\
& = \tau(x_1^2)
\left( 
\sum_{i = 1} ^{k-1} 
(-1)^{(k-1)-i} \tau(y_{i}) + (-1)^{k-1}
\right)\\
\end{align*}
hold. This proves that 
$$\tau(x_1^{2}y_{\pm k}) 
= \tau(x_1^2)
\left( 
\sum_{i = 1} ^{k_\pm}
(-1)^{k_\pm -i} \tau(y_i) + (-1)^{k_\pm}
\right).$$
The final equality follows from \cite[Lemma 2.1]{JamborLiebeckOBrien}. 

\end{proof}

\section{Non-surjectivity of word maps}

In this section, we study non-surjective word maps on $\psl_2(\bbF)$ for certain finite fields $\bbF$, and as an application, we obtain non-surjective word maps on the absolute Galois group of certain number fields. 

\subsection{Non-surjective maps on $\psl_2(\bbF)$}

For any positive integer $m$, let $\zeta_m$ denote the root of unity $e^{2\pi i /m}$. 

\begin{theorem}
\label{Thm:NonSurjWord}
Let $k$ be an integer with $k_\pm \geq 1$. Let $p$ be a prime and $n$ be a positive integer such that the following conditions hold. 
\begin{enumerate}
\item The integer $2$ is not a square modulo $p$.
\item The integer $n$ is odd. 
\item The inertia degree $f_i$ of $p$ in the extension $\bbQ(\zeta_{2k_\pm + 1}^i +  \zeta_{2k_\pm + 1}^{-i})$ does not divide $n$ for any $1\leq i \leq k_\pm$. 
\end{enumerate}
Then the word map $(x,y)\mapsto w(x, y)$ is not surjective on $\psl_2(\bbF_{p^n})$ where $w$ denotes one among the words 
$$x_1^{\pm 2} y_k, 
x_1^2y_{\pm k}
.$$
\end{theorem}

\begin{proof}
By Lemma \ref{Lemma:Factorization}, the trace polynomial of the word $w$ factors as 
$$
(s^2 -2)
\prod_{i = 1}^{k_\pm}
(\tau(y_1) + \zeta_{2k_\pm + 1}^i + \zeta_{2k_\pm + 1}^{-i})
$$
over $\bbZ[\zeta_{2k_\pm + 1}+\zeta_{2k_\pm + 1}^{-1}]$. If some element of $\sln_2(\bbF_{p^n})$ lies in the image of the induced word  map on $\sln_2(\bbF_{p^n})$, then the word polynomial $\tau(w)$ will vanish at some point of $\bbF_{p^n}^3$. The polynomial $X^2-2$ has no root in $\bbF_{p^n}$. So one of the factors of the product 
$$\prod_{i = 1}^{k_\pm}
(\tau(y_1) + \zeta_{2k\pm 1}^i + \zeta_{2k\pm 1}^{-i})
$$
with coefficients in $\bbF_{p^{nf_1 \cdots f_{k_\pm} }}^3$, 
vanishes at a point of $\bbF_{p^n}^3$. Thus $\zeta_{2k\pm 1}^i + \zeta_{2k\pm 1}^{-i}$ is contained in $\bbF_{p^n}$ for some $1\leq i \leq k_\pm$. This is impossible since $f_i$ does not divide $n$ for all $i$. This proves the result. 
\end{proof}

Note that if there exists an integer $n$ such that the condition (3) in Theorem \ref{Thm:NonSurjWord} holds, then the inertia degree of $p$ in the extension $\bbQ(\zeta_{2k_\pm + 1}^i +  \zeta_{2k_\pm + 1}^{-i})$ is at least two for any $1\leq i \leq k_\pm$. 
If $k_\pm \equiv 1 \pmod 3$, then the inertia degree of $p$ in the extension $\bbQ(\zeta_{2k_\pm + 1}^i +  \zeta_{2k_\pm + 1}^{-i})$ is one for $i = \frac{2k_\pm + 1}{3}  \leq k_\pm$. 
So, $k_\pm \not\equiv 1 \pmod 3$ is a necessary condition for having an integer $n$ satisfying the condition (3) in Theorem \ref{Thm:NonSurjWord}. We prove that this congruence condition is also sufficient to guarantee the existence of primes $p$ and integers $n\geq 1$ satisfying the conditions in Theorem \ref{Thm:NonSurjWord}. 

\begin{proposition}
\label{Prop:InfManyPrime}
Let $k$ be an integer such that $k_\pm \geq 1$ and $k_\pm \not\equiv 1 \pmod 3$. Then there are infinitely many primes $p$ such that the integer $2$ is not a square modulo $p$ and the inertia degree of $p$ in the extension $\bbQ(\zeta_{2k_\pm +1}^i + \zeta_{2k_\pm +1}^{-i})$ is at least two for any $1\leq i \leq k_\pm $. 
\end{proposition}

\begin{proof}
Note that if $p$ is prime such that $p^2 \not\equiv 1$ modulo each divisor of $2k_\pm +1$ larger than $1$, then the inertia degree of any such prime in the extension $\bbQ(\zeta_{2k_\pm +1}^i + \zeta_{2k_\pm +1}^{-i})$ is at least $2$ for any $1\leq i \leq k_\pm $. 
Thus, it suffices to show that the primes $p$ such that $2$ is not a square modulo $p$, $p$ is coprime to $2k_\pm +1$ and $p^2 \not\equiv 1$ modulo each divisor of $2k_\pm +1$ larger than $1$, have density 
$$
\frac 12 \times \prod_{p\mid 2k_\pm + 1} \left(1 - \frac 3p\right).$$
This follows since such primes are precisely the primes $p$ satisfying the following conditions: 
$$p\equiv 3, 5\pmod 8$$
and 
$p^2 \not\equiv 1$ modulo each prime divisor of $2k_\pm +1$. 
\end{proof}

\begin{corollary}
If $k$ is an integer such that $k_\pm \geq 1$ and $k_\pm \not\equiv 1 \pmod 3$, then there are infinitely many primes $p$ and integers $n\geq 1$ satisfying the conditions in Theorem \ref{Thm:NonSurjWord}. 
\end{corollary}

\subsection{Non-surjective maps on absolute Galois group of number fields}
\label{SubSec:NonSurjGalois}
It would be interesting to look at non-surjective word maps on the absolute Galois group of number fields, i.e., on the group $\mathrm{Gal}(\overline{\mathbb Q}/K)$, where $\overline{\mathbb Q}$ denotes an algebraic closure of the field of rational numbers $\mathbb Q$, and $K/\bbQ$ is a finite extension. 
The inverse Galois problem has been solved by Shih for the group $\psl_2(\bbF_p)$ for any odd prime $p$ such that $2, 3$ or $7$ is a quadratic non-residue modulo $p$ 
\cite{ShihInverseGaloisPSL2Fp}. Recently, it has been solved by Zywina for the group $\psl_2(\bbF_p)$ for any prime $p\geq 5$ \cite{ZywinaInverseGaloisPSL}. 
Thus, for any prime $p\geq 5$, there exists a finite Galois extension $L/\bbQ$ whose Galois group is isomorphic to $\psl_2(\bbF_p)$. 
We use this result to obtain the following. 

\begin{proposition}
\label{Prop:NonSurjAbsGal}
If $K$ is a number field such that its degree over $\bbQ$ is coprime to $6$, then the word map $(x,y)\mapsto w(x, y)$ is not surjective on the absolute Galois group of $K$ where $w$ denotes one among the words 
$$x_1^{\pm 2} y_k, 
x_1^2y_{\pm k}
,$$
where $k$ is an integer with $k_\pm \geq 1$ and $k_\pm \not\equiv 1 \pmod 3$. 
\end{proposition}

\begin{proof}
Let $n$ denote the degree of $K$ over $\bbQ$. Let $p\geq 5$ be a prime such that $p\equiv 2\pmod n$. Note that $p(p^2-1)\equiv 6 \pmod n$. Since $n$ is coprime to $6$, it follows that $p(p^2-1)$ is coprime to $n$. Thus, the total number of elements of $\psl_2(\bbF_p)$ is coprime to $n$. Let $L$ be a Galois extension of $\bbQ$ such that the Galois group of this extension is isomorphic to $\psl_2(\bbF_p)$. Note that the fields $K, L$ are linearly disjoint. Hence, the Galois group of the extension $KL/K$ is isomorphic to $\psl_2(\bbF_p)$. From Theorem \ref{Thm:NonSurjWord} and Proposition \ref{Prop:InfManyPrime}, it follows that the word map $(x,y)\mapsto w(x, y)$ is not surjective on $\mathrm{Gal}(\overline{\mathbb Q}/K)$ where $w$ denotes one among the above-mentioned words. 
\end{proof}

\section{Further Questions} 

The connection between the length of the word map and its surjectivity or its non-surjectivity is not yet well understood.

\begin{question}
Does there exist a non-proper-power odd length word which is non-surjective on $\psl_{2}(\mathbb{F}_{q})$ for infinitely many $q$?
\end{question}

In fact, one can ask a more refined question about the possible lengths of the words inducing non-surjective maps on $\psl_{2}(\mathbb{F}_{q})$ for infinitely many $q$.

\begin{question}
Consider the set $A$ consisting of the lengths of the word $w$ in $F_2$ such that $w$ is a non-proper-power word and is non-surjective on $\psl_2(\bbF_q)$ for infinitely many $q$.  What can we say about $|\mathbb{N}\setminus A|$ or about $\mathbb{N}\setminus A$?
\end{question}

Note that $A$ contains $3q-1$ for any prime $q\geq 5$ by \cite[Theorem 1]{JamborLiebeckOBrien}. 
From Theorem \ref{Thm:NonSurjWord} and Proposition \ref{Prop:InfManyPrime}, it follows that the set 
$A$ contains $3r-1$ for any odd integer $r\geq 5$ not divisible by $3$ and 
$A$ contains $3r-5$ for any odd integer $r\geq 7$ such that $r+1$ is not divisible by $3$. Thus, $A$ contains almost all positive integers which are congruent to $\pm 2, \pm 4 \pmod{18}$. 

\section{Acknowledgements}
The first author wishes to thank Chen Meiri for a number of discussions on word maps. The first author is supported by the ISF Grant no. 1226/19 at the Department of Mathematics at the Technion. The second author would like to acknowledge the Initiation Grant from the Indian Institute of Science Education Research Bhopal, and the INSPIRE Faculty Award from the Department of Science and Technology, Government of India.

\def\cprime{$'$} \def\Dbar{\leavevmode\lower.6ex\hbox to 0pt{\hskip-.23ex
  \accent"16\hss}D} \def\cfac#1{\ifmmode\setbox7\hbox{$\accent"5E#1$}\else
  \setbox7\hbox{\accent"5E#1}\penalty 10000\relax\fi\raise 1\ht7
  \hbox{\lower1.15ex\hbox to 1\wd7{\hss\accent"13\hss}}\penalty 10000
  \hskip-1\wd7\penalty 10000\box7}
  \def\cftil#1{\ifmmode\setbox7\hbox{$\accent"5E#1$}\else
  \setbox7\hbox{\accent"5E#1}\penalty 10000\relax\fi\raise 1\ht7
  \hbox{\lower1.15ex\hbox to 1\wd7{\hss\accent"7E\hss}}\penalty 10000
  \hskip-1\wd7\penalty 10000\box7}
  \def\polhk#1{\setbox0=\hbox{#1}{\ooalign{\hidewidth
  \lower1.5ex\hbox{`}\hidewidth\crcr\unhbox0}}}
\providecommand{\bysame}{\leavevmode\hbox to3em{\hrulefill}\thinspace}
\providecommand{\MR}{\relax\ifhmode\unskip\space\fi MR }
\providecommand{\MRhref}[2]{%
  \href{http://www.ams.org/mathscinet-getitem?mr=#1}{#2}
}
\providecommand{\href}[2]{#2}

\end{document}